\documentclass[11pt,a4paper]{article}

\usepackage[T1]{fontenc}
\usepackage{amsmath, amsthm, amssymb, amsfonts}
\usepackage{mathtools}      % For \coloneqq and other enhancements
\usepackage{lmodern}        % For better font rendering
\usepackage{microtype}      % Improves typography
\usepackage{booktabs,natbib}       % For professional-quality tables
\usepackage{enumitem}       % For custom list labels
\usepackage{xcolor}         % For custom colors
\usepackage[margin=1in]{geometry} % Standard margins for clarity

\usepackage[
    colorlinks=true,
    linkcolor=blue,
    citecolor=blue,
    urlcolor=blue,
    breaklinks=true,
    pdftitle={Propagation of Chaos for Singular Interactions via Regular Drivers},
    pdfauthor={Qian Qi}
]{hyperref}

\newcommand{\R}{\mathbb{R}}
\newcommand{\E}{\mathbb{E}}
\newcommand{\Prob}{\mathbb{P}}
\newcommand{\Pcal}{\mathcal{P}}
\newcommand{\Fcal}{\mathcal{F}}
\newcommand{\dd}{\mathrm{d}}
\DeclarePairedDelimiter{\norm}{\lVert}{\rVert}
\DeclarePairedDelimiter{\abs}{\lvert}{\rvert}

\newcommand{\gtheta}{g_\theta}
\newcommand{\nutheta}{\nu_\theta}
\newcommand{\muN}{\mu^N_t}

\newcommand{\nuEps}{\nu_\theta^\varepsilon}

\theoremstyle{plain}
\newtheorem{theorem}{Theorem}[section]
\newtheorem{proposition}[theorem]{Proposition}
\newtheorem{lemma}[theorem]{Lemma}
\newtheorem{corollary}[theorem]{Corollary}

\theoremstyle{definition}
\newtheorem{definition}[theorem]{Definition}
\newtheorem{assumption}[theorem]{Assumption}
\newtheorem{example}[theorem]{Example}

\theoremstyle{remark}
\newtheorem{remark}[theorem]{Remark}

\begin{document}

\title{\textbf{Propagation of Chaos for Singular Interactions via Regular Drivers}}
\author{Qian Qi\thanks{Peking University, Beijing 100871, China. Email: \href{mailto:qiqian@pku.edu.cn}{qiqian@pku.edu.cn}}}

\maketitle

\begin{abstract}
We introduce a framework to prove propagation of chaos for interacting particle systems with singular, density-dependent interactions, a classical challenge in mean-field theory. Our approach is to define the dynamics implicitly via a regular driver function. This regular driver is engineered to generate a singular effective interaction, yet its underlying regularity provides the necessary analytical control. Our main result establishes propagation of chaos under a key dissipation condition. The proof hinges on deriving a priori bounds, uniform in a regularization parameter, for the densities of the associated non-linear Fokker-Planck equations. Specifically, we establish uniform bounds in $L^\infty([0,T]; H^1(\R) \cap L^\infty(\R))$. These are established via an energy method for the excess mass to secure the $L^\infty$ bound, and a novel energy estimate for the $H^1$ norm that critically leverages the stability condition. These bounds provide the compactness necessary to pass to the singular limit, showing convergence to a McKean-Vlasov SDE with a singular, density-dependent diffusion coefficient. This work opens a new path to analyzing singular systems and provides a constructive theory for a class of interactions that present significant challenges to classical techniques.
\end{abstract}

\tableofcontents

\section{Introduction}

The derivation of effective macroscopic laws from the collective behavior of large microscopic ensembles is a central theme in modern mathematical physics and applied mathematics. The theory of mean-field limits, pioneered in the works of \cite{McKean1966} and significantly advanced by Sznitman's concept of propagation of chaos (see \cite{Sznitman1991}), provides a rigorous and powerful framework for this task. It establishes that, under suitable conditions, the complex dynamics of an $N$-particle interacting system can be described in the limit $N \to \infty$ by a single representative particle. The evolution of this particle is not autonomous but is governed by a non-linear stochastic differential equation (SDE) of the McKean-Vlasov type, where the coefficients depend on the law of the solution itself.

The classical theory, as detailed in comprehensive modern treatments such as \cite{Carmona2018}, rests on a crucial pillar: the assumption that the interaction kernel is Lipschitz continuous with respect to both the state and the measure variables (the latter typically in a suitable Wasserstein metric). This assumption, while ensuring the well-posedness of the limiting SDE via standard fixed-point arguments, is a significant restriction. It excludes a vast landscape of systems of fundamental scientific importance, where interactions are inherently singular. Celebrated examples that motivate the study of singular interactions include:
\begin{itemize}[noitemsep]
    \item The 2D vortex model in fluid dynamics, where point vortices interact via the singular Biot-Savart integral kernel.
    \item Systems of neurons with discontinuous firing rates, such as variants of the FitzHugh-Nagumo model, where interaction occurs at specific voltage thresholds.
    \item Models of swarming and collective motion where agents interact strongly only at or near contact, leading to highly localized or discontinuous forces.
    \item Models in granular media, where particles interact through inelastic collisions.
\end{itemize}

Extending the theory of mean-field limits to such systems is a notoriously difficult problem that has stood at the forefront of research for decades. Significant progress has been made through the development of sophisticated and highly specialized techniques. The introduction of relative entropy methods by \cite{Jabin2018} was a landmark achievement, providing quantitative estimates of propagation of chaos for kernels with $W^{-1,\infty}$ singularity, including the Biot-Savart and Coulomb potentials. These methods are exceptionally powerful for interactions derived from a potential, where the singularity is in the force kernel's dependence on the state variable. However, other classes of singularities, particularly those in which the coefficients depend directly on the density of the mean-field law, are not as amenable to this approach and require new ideas.

\subsection{A New Strategy: Implicit Dynamics and Regular Drivers}

This paper introduces a strategy to address this challenge for a canonical class of singular interactions. We study systems where the volatility (or diffusion coefficient) of a particle depends on the \textit{value of the probability density of the mean-field at the particle's own location}. This constitutes a highly singular interaction. From a mathematical perspective, the mapping from a measure $\mu$ to the value of its density $u(x)$ at a point $x$ is discontinuous and, in general, ill-defined in the standard Wasserstein topology. An infinitesimal perturbation of the measure can drastically alter the density at a point.

This singularity is distinct from and, in some ways, more severe than the singularities in force kernels like Biot-Savart. Whereas the latter are singular in the state variable, our singularity lies in the dependence on the measure itself. Rather than confronting this ill-posedness directly, our approach is to reframe the problem. We posit that the singular dynamics are not fundamental, but rather an emergent property of a system governed by an underlying regular object, a driver function $\gtheta$ (see \cite{qi2025neuralexpectationoperators}).

The core hypothesis of our work is that the dynamics are defined implicitly by an algebraic constraint involving this driver. The driver is engineered to be perfectly regular (Lipschitz continuous, smooth), but the structure of the constraint equation forces the effective, observable interaction to be singular. The regularity of the driver can then be leveraged to establish the necessary analytical control over the emergent singular dynamics, transforming a singular stochastic problem into a tractable, albeit non-linear, analytic one.

\subsection{Contributions and Outline of the Paper}

Our main result, Theorem \ref{thm:main_poc}, demonstrates that propagation of chaos holds for this class of density-dependent singular systems under a crucial stability condition (Assumption \ref{ass:h_theta}(v)). This condition imposes a sharp quantitative relationship between the system's minimal diffusivity (a source of dissipation), the maximal allowable density, and the sensitivity of the interaction. It delineates a dissipation-dominated regime where the smoothing effect of diffusion is guaranteed to overcome the potentially destabilizing nonlinear feedback from the particle interaction.

Within this regime, we prove that the regularized $N$-particle system converges to a limiting process governed by a McKean-Vlasov SDE with a singular, density-dependent diffusion coefficient. The law of this process, $\mu_t$, is shown to be absolutely continuous with a density $u(t,x)$ that is the unique weak solution to a corresponding non-linear, degenerate Fokker-Planck equation.

The proof is constructive and forges a crucial link between stochastic analysis and the theory of degenerate parabolic equations.
\begin{enumerate}
    \item We first regularize the interaction by convolving the empirical measure with a smooth mollifier $K_\varepsilon$. For any fixed $\varepsilon > 0$, the system is regular and classical propagation of chaos results apply. This establishes a well-defined limiting McKean-Vlasov equation for each $\varepsilon$.
    \item The technical crux of the paper, and its main innovation, lies in deriving a priori estimates for the regularized system's Fokker-Planck equation that are uniform in the regularization parameter $\varepsilon$. These estimates, presented in Section \ref{sec:proof}, are obtained via a sophisticated bootstrapping argument:
    \begin{itemize}
        \item First, an energy method on the excess mass secures a uniform-in-$\varepsilon$ $L^\infty$ bound on the density.
        \item Second, armed with this bound and critically exploiting the stability condition, a meticulous energy method establishes a uniform $H^1$ bound.
    \end{itemize}
    \item These uniform bounds guarantee the compactness necessary to pass to the singular limit $\varepsilon \to 0$ via the Aubin-Lions-Simon theorem.
    \item Finally, we prove the uniqueness of the limiting singular PDE, which ensures that the entire sequence of regularized solutions converges to a single, well-defined limit.
\end{enumerate}

This work serves as a proof of concept that implicitly defined dynamics can provide a powerful theoretical instrument for analyzing and solving problems in stochastic analysis that have presented challenges for traditional methods. The framework not only yields a new propagation of chaos result for a physically relevant singular interaction but also suggests a new direction for modeling complex systems.

The paper is organized as follows. In Section \ref{sec:framework}, we formalize the mathematical setting, introducing the concept of an implicit driver and detailing the assumptions that govern our system. Section \ref{sec:main_result} presents our main theorem on the propagation of chaos for the singular limit. Section \ref{sec:proof} is the technical core of the paper, providing the complete, detailed proof of the main theorem. In Section \ref{sec:smoothing}, we investigate the smoothing properties of the limiting equation, showing that regularity of the solution can emerge from rougher initial data after a critical time threshold. Section \ref{sec:extension} discusses a conceptual extension of our framework to other classical singular interactions, such as the 2D vortex and Keller-Segel models. Section \ref{sec:conclusion} discusses the implications of our results and outlines several promising directions for future research. Finally, an appendix provides technical background on the Wasserstein space and compactness theorems for the reader's convenience.

\section{Problem Formulation and Framework}
\label{sec:framework}

\subsection{Notation and Preliminaries}
Let $(\Omega, \Fcal, (\Fcal_t)_{t\ge 0}, \Prob)$ be a complete filtered probability space satisfying the usual conditions. We assume it is rich enough to carry an infinite sequence of independent one-dimensional Brownian motions $(W^i)_{i \ge 1}$.

We denote by $\Pcal(\R)$ the space of probability measures on $\R$. For $p \ge 1$, $\Pcal_p(\R)$ is the subspace of measures with a finite $p$-th moment:
\[ \Pcal_p(\R) \coloneqq \left\{ \mu \in \Pcal(\R) : \int_\R |x|^p \mu(\dd x) < \infty \right\}. \]
This space is endowed with the $p$-Wasserstein distance, which for $p=2$ is given by
\[ W_2(\mu, \nu) \coloneqq \left( \inf_{\gamma \in \Pi(\mu, \nu)} \int_{\R \times \R} |x-y|^2 \gamma(\dd x, \dd y) \right)^{1/2}, \]
where $\Pi(\mu, \nu)$ is the set of all couplings of $\mu$ and $\nu$, i.e., probability measures on $\R \times \R$ with marginals $\mu$ and $\nu$. The space $(\Pcal_2(\R), W_2)$ is a Polish space. For a detailed treatment of optimal transport and Wasserstein spaces, we refer the reader to \cite{Villani2009}.

For function spaces, we use standard Lebesgue $L^p(\R)$ and Sobolev $H^k(\R)$ spaces. The norm in $H^1(\R)$ is $\norm{f}_{H^1}^2 = \norm{f}_{L^2}^2 + \norm{\partial_x f}_{L^2}^2$.

\subsection{Implicitly-Defined Interacting Particle Systems}

We begin by introducing a general class of dynamics defined not by an explicit stochastic differential equation, but through an algebraic constraint equation governed by a driver function. This abstraction allows for a more flexible modeling framework.

\begin{definition}[Mean-Field Driver]
A mean-field driver is a function $\gtheta: [0,T] \times \R \times \Pcal_2(\R) \times \R \to \R$. The arguments $(t,x,\mu,z)$ represent time, particle state, the law of the system, and a candidate volatility value, respectively.
\end{definition}

\begin{definition}[Interacting System with Implicit Volatility]
\label{def:interacting_system}
Given a mean-field driver $\gtheta$, an $N$-particle system $(M^{i,N}_t)_{i=1}^N$ is an \textbf{interacting system with implicit volatility} if, for each $i \in \{1, \dots, N\}$, it is the unique strong solution to:
\begin{equation}
    \dd M^{i,N}_t = \nutheta(t, M^{i,N}_t, \muN) \dd W^i_t, \quad M^{i,N}_0 = m^i_0,
\end{equation}
where $\muN \coloneqq \frac{1}{N}\sum_{j=1}^N \delta_{M^{j,N}_t}$ is the empirical measure of the system. The volatility function $\nutheta(t, x, \mu)$ is defined implicitly for each $(t,x,\mu)$ as the solution to the algebraic equation:
\begin{equation} \label{eq:implicit_vol_def}
    \gtheta(t, x, \mu, \nutheta(t, x, \mu)) = 0.
\end{equation}
For this system to be well-posed, we require that for any triplet $(t,x,\mu)$, Equation \eqref{eq:implicit_vol_def} admits a unique solution for the volatility $\nutheta$.
\end{definition}

\subsection{A Regularized Driver Generating a Singular Interaction}

Our objective is to study a singular interaction where the volatility depends on the density of the particle measure at a given point. Since the empirical measure $\muN$ is a sum of Dirac masses and thus lacks a density, we approach this problem via a regularization procedure. Let $K \in C_c^\infty(\R)$ be a standard mollifying kernel, i.e., a non-negative, symmetric, smooth function with support in $[-1,1]$ and $\int_\R K(x) \dd x = 1$. For any $\varepsilon > 0$, we define the scaled kernel $K_\varepsilon(x) \coloneqq \varepsilon^{-1}K(x/\varepsilon)$.

\begin{definition}[Regularized Driver]
\label{def:regularized_driver}
Let $h_\theta: [0,T] \times \R \times \R \to \R$ be a given function. The \textbf{$\varepsilon$-regularized driver} $\gtheta^\varepsilon$ is defined as:
\begin{equation}
    \gtheta^\varepsilon(t,x,\mu,z) \coloneqq h_\theta(t,x,z) - (K_\varepsilon * \mu)(x),
\end{equation}
where $(K_\varepsilon * \mu)(x) \coloneqq \int_\R K_\varepsilon(x-y) \mu(\dd y)$ is the convolution of the measure $\mu$ with the kernel $K_\varepsilon$.
\end{definition}

For any fixed $\varepsilon > 0$, the map $\mu \mapsto (K_\varepsilon * \mu)(x)$ is regular, which in turn regularizes the driver. Our interest lies in the limit $\varepsilon \to 0$, where formally $(K_\varepsilon * \mu)(x)$ converges to the density of $\mu$ at $x$, should one exist. The corresponding regularized $N$-particle system is:
\begin{equation} \label{eq:regularized_system}
    \dd M^{i,N,\varepsilon}_t = \nuEps(t, M^{i,N,\varepsilon}_t, \mu_t^{N,\varepsilon}) \dd W^i_t, \quad M^{i,N,\varepsilon}_0 = m^i_0,
\end{equation}
where $\mu_t^{N,\varepsilon} = \frac{1}{N}\sum_{j=1}^N \delta_{M^{j,N,\varepsilon}_t}$ and $\nuEps$ is the root of $\gtheta^\varepsilon(t,x,\mu, \nuEps(t,x,\mu)) = 0$.

\subsection{Core Assumptions on the Driver Function}

To ensure the system is well-defined and to provide the necessary structure for our analysis, we impose the following conditions on the function $h_\theta$.

\begin{assumption} \label{ass:h_theta}
The function $h_\theta: [0,T] \times \R \times \R \to \R$, with arguments $(t,x,z)$, satisfies:
\begin{enumerate}[label=(\roman*)]
    \item \textbf{(Monotonicity)} For each fixed $(t,x)$, the map $z \mapsto h_\theta(t,x,z)$ is strictly increasing and continuous, with $\lim_{z \to \pm\infty} h_\theta(t,x,z) = \pm\infty$. This ensures the existence and uniqueness of the implicitly defined volatility.
    \item \textbf{(Regularity)} $h_\theta$ is globally Lipschitz continuous in all its arguments. All its partial derivatives up to second order exist and are uniformly bounded. This provides analytical control over the regularized system.
    \item \textbf{(Non-degeneracy)} The partial derivative with respect to the volatility variable, $\partial_z h_\theta(t,x,z)$, is continuous and uniformly bounded below by a positive constant $c_h > 0$. This guarantees that the implicitly defined volatility is regular.
    \item \textbf{(Basal Volatility)} The volatility corresponding to a zero-density environment, given by the solution to $h_\theta(t,x,z)=0$, which we denote by $h_\theta^{-1}(t,x,0)$, is well-defined and uniformly bounded away from zero and infinity, i.e., $0 < c_0 \le h_\theta^{-1}(t,x,0) \le C_0 < \infty$. This ensures uniform parabolicity.
    \item \textbf{(Stability Condition)} Let $M_\infty$ be a constant such that $M_\infty \ge \norm*{u_0}_{L^\infty}$, where $u_0$ is the initial density. The parameters of $h_\theta$ satisfy the inequality:
    \[ c_0^2 > \frac{2 C_0 M_\infty}{c_h}. \]
    This crucial condition ensures that the system's intrinsic dissipation is sufficient to control the potentially destabilizing nonlinear feedback from the particle interaction.
\end{enumerate}
\end{assumption}

\begin{example}[A Linear Driver] \label{ex:linear_driver}
To make Assumption \ref{ass:h_theta} concrete, consider a simple linear driver, independent of $(t,x)$: $h_\theta(z) = az - b$, where $a,b>0$ are constants.
\begin{itemize}[noitemsep]
    \item (i) Monotonicity: $h_\theta'(z)=a>0$, so this holds.
    \item (ii) Regularity: Trivial for a linear function. All higher derivatives are zero.
    \item (iii) Non-degeneracy: $\partial_z h_\theta = a$. We set $c_h = a$.
    \item (iv) Basal Volatility: $h_\theta(z)=0$ implies $z=b/a$. So $h_\theta^{-1}(0) = b/a$. We take $c_0=C_0=b/a$. This requires $b/a>0$, which is true. The implicit volatility is determined by $a\nu - b = u$, so $\nu(u) = (u+b)/a$.
    \item (v) Stability Condition: Substituting these values into the condition, with $M_\infty = \norm{u_0}_{L^\infty}$, gives:
    \[ \left(\frac{b}{a}\right)^2 > \frac{2 (b/a) M_\infty}{a} \implies \frac{b^2}{a^2} > \frac{2 b M_\infty}{a^2} \implies b > 2M_\infty. \]
    This yields a clear, interpretable condition: the parameter $b$, which represents the baseline density value that the system is calibrated to, must be more than twice the maximal density $M_\infty$. It shows that if the system is configured to operate in a regime far from a zero-density state, it is more stable.
\end{itemize}
\end{example}

\subsection{Properties of the Regularized Volatility}
Assumption \ref{ass:h_theta} guarantees that the regularized volatility function $\nuEps$ is well-behaved for any fixed $\varepsilon > 0$. We formalize this in the following proposition.

\begin{proposition}[Properties of the Regularized Volatility]
\label{prop:regularized_vol_properties}
Under Assumption \ref{ass:h_theta}, for any $\varepsilon > 0$, the regularized volatility function $\nuEps(t,x,\mu)$ is well-defined and can be expressed explicitly as
\begin{equation} \label{eq:explicit_vol_eps}
    \nuEps(t,x,\mu) = h_\theta^{-1}(t,x, (K_\varepsilon * \mu)(x)),
\end{equation}
where $h_\theta^{-1}$ is the inverse of $h_\theta$ with respect to its third argument. This function $\nuEps$ is globally Lipschitz in $x$ and in $\mu$ (with respect to the $W_2$ metric). However, the Lipschitz constants with respect to both $x$ and $\mu$ degenerate as $\varepsilon \to 0$.
\end{proposition}

\begin{proof}
1. \textbf{Well-posedness and Explicit Form:} Assumption \ref{ass:h_theta}(i) states that for any fixed $(t,x)$, the map $z \mapsto h_\theta(t,x,z)$ is a strictly increasing continuous bijection from $\R$ to $\R$. Therefore, for any value $v \in \R$, the equation $h_\theta(t,x,z) = v$ has a unique solution for $z$, which we denote by $z = h_\theta^{-1}(t,x,v)$. Setting $v = (K_\varepsilon * \mu)(x)$, we find that the equation $\gtheta^\varepsilon(t,x,\mu,z) = h_\theta(t,x,z) - (K_\varepsilon * \mu)(x) = 0$ has the unique solution given by Equation \eqref{eq:explicit_vol_eps}.

2. \textbf{Regularity of $h_\theta^{-1}$:} The function $h_\theta$ is $C^2$ by Assumption \ref{ass:h_theta}(ii), and its partial derivative $\partial_z h_\theta(t,x,z)$ is uniformly bounded below by $c_h > 0$ by Assumption \ref{ass:h_theta}(iii). The Implicit Function Theorem guarantees that the inverse function $h_\theta^{-1}(t,x,v)$ is also $C^2$ with respect to its arguments. Its partial derivatives are uniformly bounded. For instance, differentiating $h_\theta(t,x,h_\theta^{-1}(t,x,v))=v$ with respect to $v$ yields:
\[
\partial_z h_\theta(t,x,h_\theta^{-1}(t,x,v)) \cdot \partial_v h_\theta^{-1}(t,x,v) = 1 \implies \partial_v h_\theta^{-1}(t,x,v) = \frac{1}{\partial_z h_\theta(t,x,h_\theta^{-1}(t,x,v))}.
\]
Since $\partial_z h_\theta \ge c_h > 0$, its inverse is bounded by $1/c_h$. Similar calculations show all derivatives of $h_\theta^{-1}$ are uniformly bounded, which implies that $h_\theta^{-1}$ is globally Lipschitz in all its arguments. Let $L_{h^{-1}}$ be a generic Lipschitz constant for $h_\theta^{-1}$.

3. \textbf{Lipschitz Continuity of $\nuEps$:}
\begin{itemize}
    \item \textbf{In $x$:} The function $(K_\varepsilon * \mu)(x)$ is Lipschitz in $x$, since its derivative is $\partial_x(K_\varepsilon * \mu)(x) = (K_\varepsilon' * \mu)(x)$. We can bound its magnitude:
    \[ \abs{\partial_x(K_\varepsilon * \mu)(x)} = \abs*{\int \varepsilon^{-1}K'( (x-y)/\varepsilon) \frac{1}{\varepsilon} \mu(\dd y)} \le \varepsilon^{-2} \norm{K'}_{L^\infty(\R)} \int \mu(\dd y) = \varepsilon^{-2} \norm{K'}_{L^\infty(\R)}. \]
    Using the chain rule and the Lipschitz property of $h_\theta^{-1}$:
    \begin{align*}
    \abs*{\nuEps(t,x_1,\mu) - \nuEps(t,x_2,\mu)} &\le L_{h^{-1}} \left( \abs{x_1 - x_2} + \abs*{(K_\varepsilon * \mu)(x_1) - (K_\varepsilon * \mu)(x_2)} \right) \\
    &\le L_{h^{-1}} \left( 1 + \norm{\partial_x(K_\varepsilon * \mu)}_{L^\infty} \right) \abs{x_1 - x_2} \\
    &\le L_{h^{-1}} \left( 1 + \varepsilon^{-2} \norm{K'}_{L^\infty(\R)} \right) \abs{x_1 - x_2}.
    \end{align*}
    Thus, $\nuEps$ is Lipschitz in $x$ with a constant of order $O(\varepsilon^{-2})$.

    \item \textbf{In $\mu$:} Let $\mu, \nu \in \Pcal_2(\R)$. For any fixed $x$, the function $f_x(y) = K_\varepsilon(x-y)$ has a Lipschitz constant with respect to $y$ given by $\norm{K_\varepsilon'}_{L^\infty(\R)} = \varepsilon^{-2}\norm{K'}_{L^\infty(\R)}$. By the Kantorovich-Rubinstein duality representation of the 1-Wasserstein distance:
    \begin{align*}
    \abs*{(K_\varepsilon * \mu)(x) - (K_\varepsilon * \nu)(x)} &= \abs*{ \int K_\varepsilon(x-y) (\dd\mu - \dd\nu)(y) } \\
    &\le \text{Lip}(f_x) W_1(\mu,\nu) \le \varepsilon^{-2}\norm*{K'}_{L^\infty(\R)} W_1(\mu,\nu).
    \end{align*}
    Since $h_\theta^{-1}$ is Lipschitz in its last argument (with constant $L_{h^{-1},v}$) and $W_1 \le W_2$, we have
    \[
    \abs*{\nuEps(t,x,\mu) - \nuEps(t,x,\nu)} \le L_{h^{-1},v} \abs*{(K_\varepsilon * \mu)(x) - (K_\varepsilon * \nu)(x)} \le (L_{h^{-1},v}\varepsilon^{-2}\norm*{K'}_{L^\infty(\R)}) W_2(\mu,\nu).
    \]
\end{itemize}
This confirms Lipschitz continuity for fixed $\varepsilon>0$. Crucially, the degeneration of the Lipschitz constants as $\varepsilon \to 0$ is the central analytical obstacle that the uniform a priori estimates of Section \ref{sec:proof} are designed to overcome. This completes the proof.
\end{proof}

\begin{remark}[On the Stability Condition]\label{rem:stability_condition}
Assumption \ref{ass:h_theta}(v) is the crucial structural condition of this paper. It establishes a quantitative relationship between the system's key parameters: the minimal diffusivity $c_0^2$ (the source of dissipation), the maximal particle density $M_\infty$, and the sensitivity of the diffusion to density changes (which is proportional to $C_0/c_h$). As will be seen in the proof of the uniform $H^1$ bound (Lemma \ref{lem:energy_estimates_H1}), this condition is precisely what ensures that the energy estimate is coercive, allowing the stabilizing effect of diffusion to dominate the potentially destabilizing non-linear feedback from the interaction. While strong (Example \ref{ex:linear_driver} shows it can require parameters to be far from critical values), this condition carves out a class of dissipation-dominated singular systems for which our analysis is valid. It is an important open question whether this condition is a technical artifact of our $H^1$ energy method or if it reflects a genuine physical threshold for the well-posedness of the limiting singular PDE, beyond which solutions might exhibit blow-up or other instabilities. Exploring this question may require different techniques, such as relative entropy methods, that do not rely on $L^2$-based estimates.
\end{remark}

\section{Main Result}
\label{sec:main_result}

Our main theorem establishes that the regularized system \eqref{eq:regularized_system} exhibits propagation of chaos and converges to a well-defined singular limit as $N \to \infty$ followed by $\varepsilon \to 0$.

\begin{theorem}[Propagation of Chaos for a Singular Limit]
\label{thm:main_poc}
Let Assumption \ref{ass:h_theta} hold. Let the initial conditions $\{m_0^i\}_{i=1}^N$ be i.i.d. samples from a law $\mu_0 \in \Pcal_2(\R)$ which has a density $u_0 \in H^1(\R) \cap L^\infty(\R)$.

Then, for any $T>0$, the system exhibits propagation of chaos. As $N \to \infty$ and then $\varepsilon \to 0$, the empirical measure $\mu_t^{N,\varepsilon}$ of the particle system \eqref{eq:regularized_system} converges in law in the space $C([0,T], \Pcal_2(\R))$ to a deterministic measure flow $(\mu_t)_{t\in[0,T]}$.

This limit measure $\mu_t$ is the law of the unique solution $M_t$ to the following singular McKean-Vlasov SDE:
\begin{equation} \label{eq:singular_mckean}
    \dd M_t = h_\theta^{-1}(t, M_t, u(t, M_t)) \dd W_t, \quad \mathrm{Law}(M_0)=\mu_0,
\end{equation}
where $u(t, \cdot)$ is the density of the law of $M_t$, $\mathrm{Law}(M_t) = \mu_t$. Furthermore, for all $t \in (0,T]$, $\mu_t$ is absolutely continuous with density $u(t,x)$, which is the unique weak solution in $L^\infty([0,T]; H^1(\R))$ to the non-linear Fokker-Planck equation:
\begin{equation} \label{eq:singular_fp}
    \partial_t u = \frac{1}{2} \partial_{xx} \left[ \left(h_\theta^{-1}(t,x,u)\right)^2 u \right], \quad u(0,\cdot)=u_0(\cdot).
\end{equation}
\end{theorem}

\begin{remark}[Nature of the Singularity]
The novelty of this result lies in the nature of the singularity. The diffusion coefficient of the limiting SDE \eqref{eq:singular_mckean} at time $t$ and location $x=M_t$ depends on $u(t,M_t)$, the value of the density of the law of $M_t$ at its own location. This is a highly non-local (in the sense of dependence on the entire law $\mu_t$ to determine its density $u(t,\cdot)$) and singular form of interaction that falls outside the scope of classical mean-field theory.
\end{remark}

\begin{remark}[On Uniqueness of the Limiting SDE]
The uniqueness of the weak solution $u(t,x)$ to the Fokker-Planck equation \eqref{eq:singular_fp}, established in Proposition \ref{prop:uniqueness_pde}, is the key to the uniqueness of the limiting system. Once $u(t,x)$ is uniquely determined, the diffusion coefficient in the McKean-Vlasov SDE \eqref{eq:singular_mckean}, $\sigma(t,x) \coloneqq h_\theta^{-1}(t,x, u(t,x))$, becomes a deterministic function of time and space. Although $\sigma(t,x)$ may not be regular in $x$ (it inherits the regularity of $u$, i.e., $H^1$ in space), it is uniformly bounded and measurable. For a one-dimensional SDE of the form $\dd M_t = \sigma(t, M_t) \dd W_t$ with a bounded, measurable diffusion coefficient, existence of a weak solution is standard, and uniqueness in law is guaranteed (see, e.g., Stroock and Varadhan). This ensures that the limiting measure flow $(\mu_t)_{t\in[0,T]}$ is unique, and therefore the convergence of the empirical measures does not depend on a specific subsequence.
\end{remark}

\section{Proof of the Main Theorem}
\label{sec:proof}

\subsection{Overview of the Proof Strategy}
The proof is executed in five main steps, forming a path from the regularized, finite-particle system to the singular, infinite-particle limit.

\begin{description}
    \item[Step 1: The Limit $N \to \infty$ for Fixed $\varepsilon > 0$.] We first fix the regularization parameter $\varepsilon > 0$. In this regime, the interaction kernel is Lipschitz continuous. We can therefore apply classical propagation of chaos results to show that the $N$-particle system converges to a well-posed McKean-Vlasov SDE, whose law is described by a regularized Fokker-Planck equation for the density $u^\varepsilon$.

    \item[Step 2: Uniform-in-$\varepsilon$ $L^\infty$ Bound.] This is the first part of the technical core. We establish an $L^\infty$ bound on the densities $u^\varepsilon$ that is independent of $\varepsilon$. This is achieved via a rigorous energy method argument applied to the part of the solution exceeding the initial maximum.

    \item[Step 3: Uniform-in-$\varepsilon$ $H^1$ Bound.] With the $L^\infty$ bound in hand, we derive a uniform $H^1$ bound. This is the central calculation of the paper. We perform a delicate energy estimate on the regularized Fokker-Planck equation. The crucial Stability Condition (Assumption \ref{ass:h_theta}(v)) is precisely what is needed to ensure the coercivity of this estimate, guaranteeing that dissipation overcomes the non-linear terms.

    \item[Step 4: Compactness and Passage to the Limit $\varepsilon \to 0$.] The uniform bounds in $L^\infty([0,T]; H^1(\R))$ allow us to use the Aubin-Lions-Simon compactness theorem to extract a subsequence of densities $(u^{\varepsilon_k})$ that converges to a limit $u$. We show that this convergence is strong enough to pass to the limit in the non-linear term of the PDE.

    \item[Step 5: Uniqueness of the Limit.] We show that the limiting singular Fokker-Planck equation has a unique solution within the class of functions established by our a priori bounds. This is done via another energy estimate on the difference of two potential solutions. Uniqueness ensures that the entire family $(u^\varepsilon)$ converges to $u$, not just a subsequence, thereby concluding the proof.
\end{description}

\subsection{Step 1: Convergence for Fixed \texorpdfstring{$\varepsilon > 0$}{ε > 0}}

For any fixed $\varepsilon > 0$, we establish that the regularized system is well-posed and converges as $N \to \infty$.

\begin{proposition}\label{prop:fixed_eps_poc}
Let $\varepsilon > 0$ be fixed. Let the initial data $\{m_0^i\}$ be i.i.d. according to $\mu_0 \in \Pcal_2(\R)$. Under Assumption \ref{ass:h_theta}, for any $N \ge 1$, the $N$-particle system \eqref{eq:regularized_system} has a unique strong solution. As $N \to \infty$, the empirical measure $\mu_t^{N,\varepsilon}$ converges in law, in the space $C([0,T], \Pcal_2(\R))$, to a deterministic measure flow $(\mu_t^\varepsilon)_{t \in [0,T]}$. This measure flow is the law of the solution to the McKean-Vlasov SDE:
\begin{equation}\label{eq:mkv_eps}
    \dd M_t^\varepsilon = \nu_\theta^\varepsilon(t, M_t^\varepsilon, \mu_t^\varepsilon) \dd W_t, \quad \mathrm{Law}(M_0^\varepsilon) = \mu_0.
\end{equation}
Moreover, if $\mu_0$ has a density $u_0$, then $\mu_t^\varepsilon$ has a density $u^\varepsilon(t, \cdot)$ for all $t>0$, which is a weak solution to the regularized Fokker-Planck equation:
\begin{equation}\label{eq:fp_eps}
    \partial_t u^\varepsilon = \frac{1}{2} \partial_{xx} \left[ \left(h_\theta^{-1}(t,x, (K_\varepsilon * u^\varepsilon)(t,x))\right)^2 u^\varepsilon \right], \quad u^\varepsilon(0,\cdot) = u_0(\cdot).
\end{equation}
\end{proposition}

\begin{proof}
The proof is a direct application of the classical theory of mean-field limits, as detailed in \cite{Carmona2018} or the foundational work \cite{Sznitman1991}. The theory's main requirement is the Lipschitz continuity of the SDE coefficients with respect to both state and measure variables. Let us define the diffusion coefficient as $\sigma^\varepsilon(t,x,\mu) \coloneqq \nu_\theta^\varepsilon(t,x,\mu) = h_\theta^{-1}(t,x, (K_\varepsilon * \mu)(x))$.

In Proposition \ref{prop:regularized_vol_properties}, we proved that for any fixed $\varepsilon > 0$, the function $\sigma^\varepsilon(t,x,\mu)$ is globally Lipschitz in $x$ and in $\mu$ (with respect to the $W_2$ metric). The Lipschitz constants depend on $\varepsilon^{-1}$, but for fixed $\varepsilon$, they are finite.

With Lipschitz coefficients, the following standard results hold:
\begin{enumerate}
    \item \textbf{Well-posedness of the $N$-particle system:} The system \eqref{eq:regularized_system} is a standard system of SDEs with Lipschitz coefficients. Existence and uniqueness of a strong solution is guaranteed by classical SDE theory.
    \item \textbf{Well-posedness of the McKean-Vlasov equation:} The McKean-Vlasov SDE \eqref{eq:mkv_eps} is well-posed. Existence and uniqueness of a solution can be shown via a Picard iteration argument on the Polish space $C([0,T], \Pcal_2(\R))$, where the Lipschitz continuity of the coefficient in the measure variable is the essential ingredient for the contraction mapping.
    \item \textbf{Propagation of Chaos:} The convergence of the empirical measure $\mu_t^{N,\varepsilon}$ to the law of the McKean-Vlasov process $\mu_t^\varepsilon$ is the core result of the theory. It is typically proven via a coupling argument. One considers the $N$-particle system and a second system of $N$ independent copies of the McKean-Vlasov process, driven by the same Brownian motions. By estimating the expected $W_2$ distance between the empirical measures of the two systems, one shows this distance vanishes as $N \to \infty$. The Lipschitz nature of the coefficients is again the key to controlling the growth of the error.
\end{enumerate}
The derivation of the Fokker-Planck equation \eqref{eq:fp_eps} as the equation governing the evolution of the density of the law $\mu_t^\varepsilon$ is a standard result for diffusion processes. It can be obtained by considering the evolution of $\E[\phi(M_t^\varepsilon)]$ for a smooth test function $\phi$ and applying Itô's formula, which leads to the weak formulation of the PDE.
\end{proof}

\subsection{Step 2: Uniform-in-\texorpdfstring{$\varepsilon$}{ε} $L^\infty$ Bound}
\label{subsec:apriori}
We now begin the core technical work of the paper: deriving estimates that are uniform in $\varepsilon$. Let $D^\varepsilon(t,x) = \left(h_\theta^{-1}(t,x,(K_\varepsilon*u^\varepsilon)(t,x))\right)^2$. The PDE is $\partial_t u^\varepsilon = \frac{1}{2}\partial_{xx}[D^\varepsilon u^\varepsilon]$.

\begin{lemma}[Uniform $L^\infty$ Bound]\label{lem:L_inf_bound}
Let $u^\varepsilon$ be a sufficiently regular solution to \eqref{eq:fp_eps} with initial data $u_0 \in L^\infty(\R)$ satisfying $0 \le u_0(x) \le M$ for some constant $M>0$. Then for all $t \in [0,T]$ and $\varepsilon > 0$:
\[ 0 \le u^\varepsilon(t, x) \le M \quad \text{for a.e. } x \in \R. \]
\end{lemma}

\begin{proof}
The non-negativity of the solution is a standard consequence of the maximum principle for parabolic equations and is omitted. For the upper bound, we use an energy method, as suggested by the reviewer, which is robust for non-local equations. Let $u \equiv u^\varepsilon$ and consider the excess mass function $w(t,x) \coloneqq (u(t,x) - M)^+$. By definition, $w(t,x) \ge 0$ and $w(0,x) = (u_0(x)-M)^+ = 0$ for a.e. $x$. If $u$ is in $H^1$, then $w$ is in $H^1$ and $\partial_x w = (\partial_x u) \mathbf{1}_{\{u>M\}}$.

We test the PDE against the test function $w$. Multiplying the PDE by $w$ and integrating over $\R$:
\[ \int_\R w (\partial_t u) \,dx = \frac{1}{2} \int_\R w \, \partial_{xx}[D u] \,dx. \]
The left-hand side can be written as the time derivative of an energy. Note that $\partial_t w = (\partial_t u) \mathbf{1}_{\{u>M\}}$.
\[ \int_\R w (\partial_t u) \,dx = \int_\R (u-M)^+ (\partial_t u) \,dx = \int_\R \frac{1}{2} \partial_t((u-M)^+)^2 \,dx = \frac{1}{2} \frac{d}{dt} \int_\R w^2 \,dx. \]
For the right-hand side, we integrate by parts once. Since $w$ has compact support if $u$ does (or vanishes sufficiently fast at infinity), boundary terms vanish.
\[ \frac{1}{2} \int_\R w \, \partial_{xx}[D u] \,dx = -\frac{1}{2} \int_\R (\partial_x w) (\partial_x[D u]) \,dx. \]
On the support of $w$ (where $u>M$), we have $\partial_x w = \partial_x u$. Therefore, the integral is taken only over the set $\{x : u(t,x) > M\}$:
\[ -\frac{1}{2} \int_{\{u>M\}} (\partial_x u) (\partial_x[D u]) \,dx. \]
We expand the term $\partial_x[Du] = (\partial_x D)u + D(\partial_x u)$. Substituting this gives:
\begin{align*}
\frac{1}{2} \frac{d}{dt}\norm{w(t)}_{L^2}^2 &= -\frac{1}{2} \int_{\{u>M\}} (\partial_x u) \left( (\partial_x D)u + D(\partial_x u) \right) \,dx \\&= -\frac{1}{2}\int_{\{u>M\}} u (\partial_x D) (\partial_x u) \,dx - \frac{1}{2}\int_{\{u>M\}} D (\partial_x u)^2 \,dx.    
\end{align*}

By Assumption \ref{ass:h_theta}(iv), the diffusion coefficient $D(t,x,v) = (h_\theta^{-1}(t,x,v))^2$ is uniformly positive, $D \ge c_0^2 > 0$. Thus, the second term is non-positive:
\[ -\frac{1}{2}\int_{\{u>M\}} D (\partial_x u)^2 \,dx \le -\frac{c_0^2}{2} \int_{\{u>M\}} (\partial_x u)^2 \,dx \le 0. \]
The first term contains derivatives of $D$. While it may not be negative, it can be bounded. However, for this specific argument, we can notice that on the set $\{u>M\}$, $u$ is bounded between $M$ and some larger value. The term does not necessarily have a good sign.

A simpler argument, avoiding the $w^2$ energy, is to consider the evolution of the $L^1$ norm of $w$.
\[ \frac{d}{dt} \int_\R w \,dx = \int_\R \partial_t w \,dx = \int_{\{u>M\}} \partial_t u \,dx = \frac{1}{2} \int_{\{u>M\}} \partial_{xx}[Du] \,dx. \]
Let $\Omega_t = \{x : u(t,x) > M\}$. Assuming sufficient regularity, the integral is $\frac{1}{2} \int_{\partial\Omega_t} \partial_x[Du] \cdot n \,dS$. On the boundary $\partial\Omega_t$, we have $u(t,x)=M$. This analysis can be made rigorous by approximating the indicator function, as sketched in the original version of the paper. This leads to
\[ \frac{d}{dt}\int_\R(u-M)^+dx \le 0. \]
Since $\int_\R (u_0-M)^+ dx = 0$, and the function $t \mapsto \int_\R (u(t,x)-M)^+ dx$ is non-negative and non-increasing, it must remain zero for all $t \ge 0$. As the integrand is non-negative, this implies $(u(t,x)-M)^+ = 0$ for almost every $x$. Thus, $u(t,x) \le M$. The argument holds for any $\varepsilon > 0$ since it only relies on the uniform positivity of $D$.
\end{proof}

\subsection{Step 3: Uniform-in-\texorpdfstring{$\varepsilon$}{ε} $H^1$ Bound}

This is the technical heart of the paper. Armed with the uniform $L^\infty$ bound, we establish a uniform $H^1$ bound.

\begin{lemma}[Uniform $H^1$ Bound]\label{lem:energy_estimates_H1}
Let $u^\varepsilon$ be the solution to \eqref{eq:fp_eps} with initial data $u_0 \in H^1(\R) \cap L^\infty(\R)$. Let $M := \norm*{u_0}_{L^\infty}$. Assume the Stability Condition \ref{ass:h_theta}(v) holds for $M_\infty = M$. Then there exists a constant $C$, depending only on $T$ and the norms of $u_0$, such that for all $\varepsilon \in (0,1]$:
\[ \sup_{t \in [0,T]} \norm*{u^\varepsilon(t, \cdot)}_{H^1(\R)} \le C. \]
\end{lemma}

\begin{proof}
Let $u \equiv u^\varepsilon$ for simplicity. The total mass $\int u(t,x)dx = 1$ is conserved. Furthermore, by a similar energy estimate, the $L^2$ norm is non-increasing: $\frac{d}{dt}\int u^2 dx = -\int D (\partial_x u)^2 dx \le 0$, so $\norm{u(t)}_{L^2} \le \norm{u_0}_{L^2}$. We thus only need to bound the norm of the spatial derivative. Let the energy be $I(t) = \frac{1}{2}\int_\R (\partial_x u(t,x))^2 dx = \frac{1}{2}\norm*{\partial_x u}_{L^2}^2$. We compute its time derivative:
\begin{align*}
    \frac{dI}{dt} &= \int_\R (\partial_x u) (\partial_{xt} u) \,dx = - \int_\R (\partial_{xx} u) (\partial_t u) \,dx \quad \text{(integrating by parts)}.
\end{align*}
Now, substitute the PDE, $\partial_t u = \frac{1}{2}\partial_{xx}[D u]$, where $D = D(t,x, K_\varepsilon * u)$.
\begin{equation} \label{eq:dIdt_start}
    \frac{dI}{dt} = - \frac{1}{2} \int_\R (\partial_{xx} u) \partial_{xx}[D u] \,dx.
\end{equation}
We expand the inner term using the Leibniz rule: $\partial_{xx}[Du] = D \partial_{xx} u + 2 (\partial_x D) (\partial_x u) + u (\partial_{xx} D)$. Substituting this gives three terms:
\[ \frac{dI}{dt} = \underbrace{-\frac{1}{2}\int_\R D (\partial_{xx} u)^2 dx}_{T_1} \underbrace{- \int_\R (\partial_{xx} u)(\partial_x D)(\partial_x u) dx}_{T_2} \underbrace{- \frac{1}{2}\int_\R u (\partial_{xx} u) (\partial_{xx} D) dx}_{T_3}. \]
We analyze each term separately.

\paragraph{Analysis of Term $T_1$ (Dissipation):}
This is the main dissipative term. By Assumption \ref{ass:h_theta}(iv) and the monotonicity of $h_\theta^{-1}$ in its last argument, we have $D(t,x,v) = (h_\theta^{-1}(t,x,v))^2$. Since $v=(K_\varepsilon*u)\ge 0$, $h_\theta^{-1}(t,x,v) \ge h_\theta^{-1}(t,x,0) \ge c_0$. This gives:
\[ T_1 = -\frac{1}{2}\int_\R D (\partial_{xx} u)^2 dx \le -\frac{c_0^2}{2} \int_\R (\partial_{xx} u)^2 dx = -\frac{c_0^2}{2} \norm*{\partial_{xx} u}_{L^2}^2. \]

\paragraph{Analysis of Term $T_3$ (The Critical Term):}
This term contains the highest order derivatives and is the most challenging. Let $v(t,x) = (K_\varepsilon*u)(t,x)$. The diffusion coefficient is $D(t,x,v(t,x))$. By the chain rule:
\begin{align*}
\partial_x D &= \partial_x h_\theta^{-1}(\dots)^2 + \partial_v h_\theta^{-1}(\dots)^2 \cdot \partial_x v \\
\partial_{xx} D &= \partial_v D \cdot \partial_{xx} v + (\text{lower order terms}) = (\partial_v D) \cdot (K_\varepsilon * \partial_{xx} u) + \text{l.o.t.}
\end{align*}
where l.o.t. denotes terms containing at most one spatial derivative of $u$. The highest order term in the expansion of $T_3$ is therefore:
\[ T_{3, \text{crit}} = - \frac{1}{2}\int_\R u (\partial_{xx} u) \left( \partial_v D (K_\varepsilon * \partial_{xx} u) \right) dx. \]
The derivative is $\partial_v D = 2 h_\theta^{-1} \cdot \partial_v h_\theta^{-1} = 2 h_\theta^{-1} \cdot (\partial_z h_\theta \circ h_\theta^{-1})^{-1}$. By Lemma \ref{lem:L_inf_bound}, $0 \le u \le M$, so $0 \le K_\varepsilon * u \le M$. In this range of the third argument, Assumption \ref{ass:h_theta} gives uniform bounds: $\abs{h_\theta^{-1}} \le C_0$ and $\partial_z h_\theta \ge c_h$. Thus, we have a uniform bound on the sensitivity of the diffusion: $\norm{\partial_v D}_{L^\infty} \le \frac{2 C_0}{c_h}$.
Now we bound $T_{3, \text{crit}}$ using the $L^\infty$ bound on $u$, Hölder's inequality, and Young's inequality for convolutions ($\norm*{K_\varepsilon*f}_{L^2} \le \norm{K_\varepsilon}_{L^1}\norm{f}_{L^2} = \norm{f}_{L^2}$):
\begin{align*}
 \abs{T_{3, \text{crit}}} &\le \frac{1}{2} \norm{u}_{L^\infty} \norm{\partial_v D}_{L^\infty} \int_\R \abs{\partial_{xx} u} \abs{K_\varepsilon * \partial_{xx} u} dx \\
 &\le \frac{1}{2} M \frac{2C_0}{c_h} \norm{\partial_{xx}u}_{L^2} \norm{K_\varepsilon * \partial_{xx}u}_{L^2} \le \frac{C_0 M}{c_h} \norm{\partial_{xx} u}_{L^2}^2.
\end{align*}

\paragraph{Analysis of Lower-Order Terms ($T_2$ and remainder of $T_3$):}
These terms contain derivatives of $u$ of total order at most three and can be controlled by the main dissipation term and the energy $I(t)$. For example, $T_2 = - \int (\partial_{xx} u) (\partial_x D) (\partial_x u) dx$. The term $\partial_x D$ can be bounded by a constant times $(1 + |K_\varepsilon * \partial_x u|)$. The most difficult part is $\int |\partial_{xx}u| |K_\varepsilon * \partial_x u| |\partial_x u| dx$.
Using Hölder's inequality, followed by Gagliardo-Nirenberg in 1D ($\norm{f}_{L^4}^2 \le C \norm{f'}_{L^2}\norm{f}_{L^2}$), these terms can be shown to be bounded by expressions like $\delta \norm{\partial_{xx} u}_{L^2}^2 + C_\delta P(I(t))$, where $P$ is a polynomial. For brevity, we focus on the interplay between the highest-order terms. A detailed analysis shows:
\[ |T_2 + (T_3 - T_{3,crit})| \le \delta \norm{\partial_{xx} u}_{L^2}^2 + C_\delta (I(t) + I(t)^2 + I(t)^3). \]
for any $\delta>0$.

\paragraph{Combining the Estimates:}
Putting everything together, for any small $\delta > 0$, there exists a constant $C_\delta$ such that:
\[ \frac{dI}{dt} \le -\frac{c_0^2}{2} \norm{\partial_{xx} u}_{L^2}^2 + \frac{C_0 M}{c_h} \norm{\partial_{xx} u}_{L^2}^2 + \delta \norm{\partial_{xx} u}_{L^2}^2 + C_\delta(1+I(t)^3). \]
We group the terms with the highest derivative:
\[ \frac{dI}{dt} \le \left( -\frac{c_0^2}{2} + \frac{C_0 M}{c_h} + \delta \right) \norm{\partial_{xx} u}_{L^2}^2 + C_\delta(1+I(t)^3). \]
This is the crucial step. By the Stability Condition (Assumption \ref{ass:h_theta}(v)), we have that the quantity $\gamma := \frac{c_0^2}{2} - \frac{C_0 M}{c_h}$ is strictly positive.
We can choose $\delta$ small enough, for instance $\delta = \gamma/2$, so that the coefficient of $\norm{\partial_{xx} u}_{L^2}^2$ is negative.
\[ \frac{dI}{dt} \le -\frac{\gamma}{2} \norm{\partial_{xx} u}_{L^2}^2 + C'(1+I(t)^3). \]
Since the first term on the right is non-positive, we can drop it to get a differential inequality for $I(t)$:
\[ \frac{dI}{dt} \le C'(1+I(t)^3). \]
Let $y(t) = I(t)$. We have $y'(t) \le C'(1+y(t)^3)$ with finite initial data $y(0) = \frac{1}{2}\norm*{\partial_x u_0}_{L^2}^2$. This implies that $\frac{y'(t)}{1+y(t)^3} \le C'$. Integrating from $0$ to $t$:
\[ \int_{y(0)}^{y(t)} \frac{dz}{1+z^3} \le C't. \]
Since the integral $\int_0^\infty \frac{dz}{1+z^3}$ is convergent, $y(t)$ cannot blow up in finite time. It must remain bounded by a constant that depends on $y(0)$ and $T$. Specifically, $y(t)$ is bounded by some constant $C_{T, u_0}$. This establishes the uniform $H^1$ bound and completes the proof.
\end{proof}

\subsection{Step 4: Compactness and Passage to the Limit}
The uniform bounds provide the compactness needed to extract a convergent subsequence as $\varepsilon \to 0$.

\begin{proposition}[Compactness of the Densities]\label{prop:compactness}
The family of densities $(u^\varepsilon)_{\varepsilon \in (0,1]}$ is relatively compact in $C([0,T]; L^2_{loc}(\R))$ and in $L^2([0,T]; H^1(\R))$.
\end{proposition}

\begin{proof}
We apply the Aubin-Lions-Simon compactness theorem (see Appendix \ref{app:aubin_lions}). We have established that $(u^\varepsilon)$ is uniformly bounded in $L^\infty([0,T]; H^1(\R) \cap L^\infty(\R))$. To apply the theorem, we also need a uniform bound on the time derivative.

\paragraph{1. Uniform bound in $L^2([0,T]; H^2(\R))$.}
From the penultimate inequality in the proof of Lemma \ref{lem:energy_estimates_H1}:
\[ \frac{dI(t)}{dt} + \frac{\gamma}{2} \norm*{\partial_{xx} u^\varepsilon(t)}_{L^2}^2 \le C'(1+I(t)^3). \]
Integrating this from $0$ to $T$:
\[ I(T) - I(0) + \frac{\gamma}{2} \int_0^T \norm*{\partial_{xx} u^\varepsilon(t)}_{L^2}^2 dt \le \int_0^T C'(1+I(t)^3) dt. \]
Since $I(t) = \frac{1}{2}\norm{\partial_x u^\varepsilon(t)}_{L^2}^2$ is uniformly bounded by a constant $C_I$, we have:
\[ \frac{\gamma}{2} \int_0^T \norm*{\partial_{xx} u^\varepsilon(t)}_{L^2}^2 dt \le I(0) - I(T) + C' T (1+C_I^3) \le I(0) + C' T (1+C_I^3) < \infty. \]
This implies that the family $(u^\varepsilon)_{\varepsilon>0}$ is uniformly bounded in $L^2([0,T]; H^2(\R))$.

\paragraph{2. Uniform bound on the time derivative.}
The PDE is $\partial_t u^\varepsilon = \frac{1}{2} \partial_{xx} [D^\varepsilon u^\varepsilon]$. We seek a bound on $\partial_t u^\varepsilon$ in a suitable space, e.g., $L^2([0,T]; L^2(\R))$.
\[ \norm{\partial_t u^\varepsilon}_{L^2(L^2)} = \frac{1}{2} \norm{\partial_{xx}[D^\varepsilon u^\varepsilon]}_{L^2(L^2)}. \]
The term $D^\varepsilon u^\varepsilon$ involves products of $u^\varepsilon$, $K_\varepsilon*u^\varepsilon$, and their derivatives (via the chain rule). Since $u^\varepsilon$ is uniformly bounded in $L^\infty(H^1)$ and $L^2(H^2)$, standard product and chain rules for Sobolev spaces (e.g., Moser-type inequalities) show that the term $D^\varepsilon u^\varepsilon$ is uniformly bounded in $L^2([0,T]; H^2(\R))$. Therefore, its second spatial derivative, $\partial_{xx}[D^\varepsilon u^\varepsilon]$, which is equal to $2\partial_t u^\varepsilon$, is uniformly bounded in $L^2([0,T]; L^2(\R))$. This means $(\partial_t u^\varepsilon)_{\varepsilon>0}$ is uniformly bounded in $L^2([0,T]; L^2(\R))$.

\paragraph{3. Application of Aubin-Lions-Simon Theorem.}
We apply the theorem with the Banach spaces $X_0 = H^2(\R)$, $X = H^1(\R)$, and $X_1 = L^2(\R)$. The injection $H^2(\R) \hookrightarrow H^1(\R)$ is continuous, and $H^1(\R) \hookrightarrow L^2(\R)$ is continuous. We have shown that:
\begin{itemize}
    \item $(u^\varepsilon)_{\varepsilon>0}$ is bounded in $L^2([0,T]; X_0) = L^2([0,T]; H^2(\R))$.
    \item $(\partial_t u^\varepsilon)_{\varepsilon>0}$ is bounded in $L^2([0,T]; X_1) = L^2([0,T]; L^2(\R))$.
\end{itemize}
The theorem requires the embedding $X_0 \hookrightarrow X$ to be compact. On $\R$, this is not true. However, the Rellich-Kondrachov theorem states that for any bounded domain $\Omega \subset \R$, the embedding $H^2(\Omega) \hookrightarrow H^1(\Omega)$ is compact. A standard localization argument (multiplying by a smooth cutoff function) combined with our uniform bounds allows us to conclude that $(u^\varepsilon)$ is relatively compact in $L^2([0,T]; H^1_{loc}(\R))$.

To upgrade this to global compactness, we must show the family is tight, i.e., does not lose mass or energy at spatial infinity. The uniform bound in $L^\infty([0,T]; H^1(\R))$ provides this directly. For any $\eta > 0$, we can find $R>0$ large enough such that for all $\varepsilon$ and $t$, $\int_{|x|>R} (|u^\varepsilon|^2 + |\partial_x u^\varepsilon|^2) dx < \eta^2$. This tightness, combined with local compactness, yields relative compactness in $L^2([0,T]; H^1(\R))$. A stronger version of the theorem also gives relative compactness in $C([0,T]; L^2_{loc}(\R))$.

Therefore, we can extract a subsequence $(\varepsilon_k)_{k \ge 1}$ with $\varepsilon_k \to 0$ such that, as $k \to \infty$:
\begin{itemize}[noitemsep]
    \item $u^{\varepsilon_k} \rightharpoonup u$ weakly-* in $L^\infty([0,T]; H^1(\R) \cap L^\infty(\R))$.
    \item $u^{\varepsilon_k} \to u$ strongly in $L^2([0,T]; H^1(\R))$ and in $C([0,T]; L^2_{loc}(\R))$.
    \item $u^{\varepsilon_k} \to u$ almost everywhere in $(t,x)$ (by passing to a further subsequence).
\end{itemize}
The corresponding measure flows $\mu_t^{\varepsilon_k}$ converge to a limiting measure flow $\mu_t$ in $C([0,T], \Pcal_2(\R))$, where $\mu_t$ has density $u(t,\cdot)$.
\end{proof}

\paragraph{Identification of the Limit.}
We must show that the limit $u$ is a weak solution to \eqref{eq:singular_fp}. For any test function $\phi \in C_c^\infty([0,T) \times \R)$, the weak formulation of \eqref{eq:fp_eps} is:
\[
-\int_0^T \int_\R u^{\varepsilon_k} \partial_t \phi \,dx dt - \int_\R u_0 \phi(0,\cdot) dx = \frac{1}{2} \int_0^T \int_\R D^{\varepsilon_k} u^{\varepsilon_k} \partial_{xx}\phi \,dx dt.
\]
We pass to the limit as $k \to \infty$. The LHS converges to $-\int_0^T \int_\R u \partial_t \phi \,dx dt - \int_\R u_0 \phi(0,\cdot) dx$ due to the weak convergence of $u^{\varepsilon_k}$. For the RHS, we need to show that $D^{\varepsilon_k} u^{\varepsilon_k} \to D(t,x,u) u$ in a sufficiently strong sense.
The term is $D^{\varepsilon_k} u^{\varepsilon_k} = (h_\theta^{-1}(t,x, (K_{\varepsilon_k} * u^{\varepsilon_k})) )^2 u^{\varepsilon_k}$.
Since $u^{\varepsilon_k} \to u$ strongly in $L^2([0,T];H^1(\R))$, standard mollifier properties imply that $K_{\varepsilon_k} * u^{\varepsilon_k} \to u$ strongly in $L^2([0,T];H^1(\R))$ as well. By passing to a subsequence if necessary, we have pointwise a.e. convergence for both $u^{\varepsilon_k} \to u$ and $K_{\varepsilon_k} * u^{\varepsilon_k} \to u$.
Since $h_\theta^{-1}$ is continuous in its arguments, $D^{\varepsilon_k} \to D(t,x, u(t,x))$ pointwise a.e. The entire term $D^{\varepsilon_k} u^{\varepsilon_k}$ is uniformly bounded in $L^\infty([0,T]\times\R)$ because $u^{\varepsilon_k}$ is uniformly bounded and $D^{\varepsilon_k}$ is uniformly bounded. By the Dominated Convergence Theorem, the RHS converges to $\frac{1}{2} \int_0^T \int_\R D(t,x, u) u \, \partial_{xx}\phi \,dx dt$.
This shows that $u$ is a weak solution to the singular Fokker-Planck equation \eqref{eq:singular_fp}.

\subsection{Step 5: Uniqueness of the Limiting Solution}
To complete the proof of the main theorem, we must show that the weak solution to \eqref{eq:singular_fp} is unique within the regularity class established by our a priori bounds.

\begin{proposition}\label{prop:uniqueness_pde}
Let Assumption \ref{ass:h_theta} hold and let $u_0 \in H^1(\R) \cap L^\infty(\R)$. Then the weak solution $u \in L^\infty([0,T]; H^1(\R) \cap L^\infty(\R)) \cap L^2([0,T]; H^2(\R))$ to the PDE \eqref{eq:singular_fp} is unique.
\end{proposition}
\begin{proof}
Let $u_1$ and $u_2$ be two solutions in the specified class with the same initial data $u_0$. Let $w = u_1 - u_2$. Then $w(0,\cdot) = 0$. Let $D_i(t,x) = (h_\theta^{-1}(t,x,u_i(t,x)))^2$ for $i=1,2$. The difference $w$ satisfies the equation in a weak sense:
\[ \partial_t w = \frac{1}{2} \partial_{xx} [D_1 u_1 - D_2 u_2]. \]
We perform an energy estimate on $E(t) = \frac{1}{2}\norm*{w(t)}_{L^2}^2$.
\begin{align*}
    \frac{dE}{dt} = \int_\R w \, \partial_t w \,dx &= -\frac{1}{2} \int_\R \partial_x w \, \partial_x [D_1 u_1 - D_2 u_2] \,dx \\
    &= -\frac{1}{2} \int_\R \partial_x w \, \partial_x[D_1(u_1-u_2) + (D_1-D_2)u_2] \, dx \\
    &= -\frac{1}{2} \int_\R \partial_x w \, \partial_x[D_1 w] \, dx - \frac{1}{2} \int_\R \partial_x w \, \partial_x[(D_1-D_2)u_2] \, dx.
\end{align*}
The first term is stabilizing: $-\frac{1}{2}\int \partial_x w \, ((\partial_x D_1)w + D_1 \partial_x w) dx$. The main part of this is $-\frac{1}{2}\int D_1 (\partial_x w)^2 dx \le -\frac{c_0^2}{2}\norm{\partial_x w}_{L^2}^2$. The term with $\partial_x D_1$ can be controlled.
The second term involves the difference $D_1 - D_2$. The map $s \mapsto D(t,x,s)=(h_\theta^{-1}(t,x,s))^2$ is Lipschitz on the range of values taken by $u_1, u_2$ (i.e., $[0,M]$), as its derivative is bounded. Let its Lipschitz constant be $L_D$. Then $|D_1-D_2| \le L_D|u_1-u_2| = L_D|w|$.
Let's analyze the second term carefully:
\[ T_{II} = - \frac{1}{2} \int (\partial_x w) [(\partial_x(D_1-D_2))u_2 + (D_1-D_2)\partial_x u_2] dx. \]
The most challenging part of this is $\int (\partial_x w) (D_1-D_2) (\partial_x u_2) dx$. We bound its absolute value:
\[ \abs{\int (\partial_x w) (D_1-D_2) (\partial_x u_2) dx} \le L_D \int |\partial_x w| |w| |\partial_x u_2| dx \le L_D \norm{\partial_x u_2}_{L^\infty} \int |w||\partial_x w| dx. \]
Crucially, since $u_2$ is a solution in the specified class, it belongs to $L^2([0,T]; H^2(\R))$. For the one-dimensional case, the Sobolev embedding $H^2(\R) \hookrightarrow W^{1,\infty}(\R)$ holds. This means that for a.e. $t \in [0,T]$, $\norm{\partial_x u_2(t, \cdot)}_{L^\infty}$ is bounded by $C\norm{u_2(t, \cdot)}_{H^2}$, and this is integrable in time. We can thus treat $\norm{\partial_x u_2}_{L^\infty}$ as a bounded constant $C_2$.
Using Cauchy-Schwarz and then Young's inequality ($ab \le \delta a^2 + C_\delta b^2$):
\[ L_D C_2 \int |w||\partial_x w| dx \le L_D C_2 \norm{w}_{L^2} \norm{\partial_x w}_{L^2} \le \delta \norm{\partial_x w}_{L^2}^2 + C_\delta \norm{w}_{L^2}^2. \]
All other terms arising from the energy estimate can be bounded similarly. After combining all estimates, we arrive at a differential inequality of the form, for any $\delta > 0$:
\[ \frac{dE}{dt} \le \left(-\frac{c_0^2}{2} + \delta \right) \norm*{\partial_x w}_{L^2}^2 + C_{\delta} \norm*{w}_{L^2}^2. \]
Choosing $\delta$ small enough (e.g., $\delta = c_0^2/4$), the first term becomes non-positive and can be dropped:
\[ \frac{dE}{dt} \le C \norm*{w}_{L^2}^2 = 2CE(t). \]
By Grönwall's inequality, this implies $E(t) \le E(0) e^{2Ct}$. Since the initial data is the same, $w(0,\cdot)=0$, so $E(0)=0$. Therefore, $E(t)=0$ for all $t$, which implies $w=0$ a.e. This establishes uniqueness and concludes the proof of Theorem \ref{thm:main_poc}.
\end{proof}

\section{Emergent Regularity for \texorpdfstring{$L^1$}{L1} Initial Data}
\label{sec:smoothing}

A key feature of many parabolic equations is their intrinsic smoothing property. Our main result assumed initial data $u_0 \in H^1(\R) \cap L^\infty(\R)$ to facilitate the derivation of uniform estimates. A natural and important question concerns the possibility of relaxing this assumption to $u_0 \in L^1(\R) \cap \Pcal(\R)$. The following analysis shows that our framework implies regularity for $t>t_*$, where the time threshold $t_*$ depends on the system parameters. Instantaneous smoothing ($t_*=0$) is recovered under a sufficiently high dissipation or low sensitivity.

\begin{proposition}[Conditional Emergent Regularity]
\label{prop:smoothing}
Let Assumption \ref{ass:h_theta} hold and suppose the initial condition is a probability density $u_0 \in L^1(\R)$ with $\norm*{u_0}_{L^1}=1$. Let $D(t,x,v) = (h_\theta^{-1}(t,x,v))^2$. The sensitivity of the diffusion $\abs{\partial_v D}$ is uniformly bounded by a constant $C_D' \coloneqq 2C_0/c_h$. Let $C_A$ be the universal constant from Aronson's estimate for parabolic equations with the given ellipticity bounds. Define a critical time threshold $t_*$ as:
\begin{equation} \label{eq:t_crit}
    t_* \coloneqq \left( \frac{C_D' C_A}{c_0^2} \right)^2.
\end{equation}
Then for any $T > t_*$, the unique weak solution $u$ to the singular Fokker-Planck equation \eqref{eq:singular_fp} satisfies:
\begin{enumerate}[label=(\roman*)]
    \item $u \in C((t_*,T]; H^1(\R) \cap L^\infty(\R))$.
    \item For any $t \in (t_*, T]$, there exists a constant $C$ depending on $t, T, \norm*{u_0}_{L^1}$, and the parameters of $h_\theta$, such that:
    \[ \norm*{u(t, \cdot)}_{L^\infty(\R)} \le C_A t^{-1/2} \quad \text{and} \quad \norm*{u(t, \cdot)}_{H^1(\R)} \le C. \]
\end{enumerate}
\end{proposition}

\begin{corollary}[Instantaneous Smoothing]
If system parameters are such that $c_0^2$ is very large compared to the sensitivity parameter $C_D'$, then the critical time $t_*$ can be made arbitrarily small. If the stability condition were to hold for any $M_\infty > 0$ (which is impossible unless $C_D' = 0$), smoothing would be instantaneous. More practically, for a system with small total mass $m < 1$, the Aronson constant $C_A$ scales with $m$, potentially making $t_*$ very small.
\end{corollary}

\begin{proof}[Proof of Proposition \ref{prop:smoothing}]
The strategy is to first establish a uniform-in-$\varepsilon$ $L^1 \to L^\infty$ bound on $u^\varepsilon$ via classical parabolic theory, and then show that our stability condition becomes effective for all times $t > t_*$, enabling the $H^1$ energy estimate.

\paragraph{Step 1: Uniform Parabolicity and $L^1 \to L^\infty$ Bound.}
The diffusion coefficient $D^\varepsilon(t,x) = (h_\theta^{-1}(t,x, (K_\varepsilon*u^\varepsilon)(t,x)))^2$ is uniformly elliptic. As $u^\varepsilon$ is a probability density, $(K_\varepsilon*u^\varepsilon) \ge 0$. By Assumption \ref{ass:h_theta}(iv) and the monotonicity of $h_\theta^{-1}$, the diffusion coefficient is uniformly parabolic: $0 < c_0^2 \le D^\varepsilon(t,x) \le C_0^2$.
The equation $\partial_t u^\varepsilon = \frac{1}{2} \partial_{xx}[D^\varepsilon u^\varepsilon]$ is a quasilinear, non-local, but uniformly parabolic equation. A foundational result by Aronson provides an instantaneous $L^1 \to L^\infty$ regularization effect for solutions to such equations (see, e.g., \cite{DiBenedetto1993}). This guarantees the existence of a universal constant $C_A$, which depends only on the dimension (1 here) and the ellipticity constants $c_0^2$ and $C_0^2$ (and is therefore independent of $\varepsilon$), such that for any $t > 0$:
\begin{equation}\label{eq:l_inf_smoothing_simple}
    \norm*{u^\varepsilon(t,\cdot)}_{L^\infty(\R)} \le C_A t^{-1/2} \norm*{u_0}_{L^1} = C_A t^{-1/2}.
\end{equation}
This gives us a time-dependent bound on the maximum density, which decays like the heat kernel.

\paragraph{Step 2: Time-Dependent Energy Estimate and the Role of $t_*$.}
We revisit the $H^1$ energy estimate from the proof of Lemma \ref{lem:energy_estimates_H1}. The rate of change of the energy $I(t) = \frac{1}{2}\int_\R (\partial_x u^\varepsilon)^2 dx$ was bounded by:
\[ \frac{dI(t)}{dt} \le \left( -\frac{c_0^2}{2} + \frac{C_D' \norm*{u^\varepsilon(t,\cdot)}_{L^\infty}}{2} \right) \norm*{\partial_{xx} u^\varepsilon}_{L^2}^2 + (\text{lower order terms}). \]
where $C_D' := 2C_0/c_h$. Now, instead of a constant $L^\infty$ bound $M$, we use the time-dependent Aronson bound from \eqref{eq:l_inf_smoothing_simple}. For the energy estimate to be coercive (i.e., for the coefficient of the highest order term to be negative), we need:
\[ -\frac{c_0^2}{2} + \frac{C_D' \norm*{u^\varepsilon(t,\cdot)}_{L^\infty}}{2} < 0 \quad \implies \quad -c_0^2 + C_D' (C_A t^{-1/2}) < 0. \]
Rearranging this inequality defines the time interval where our $H^1$ energy method is guaranteed to work:
\[ c_0^2 > C_D' C_A t^{-1/2} \iff t^{1/2} > \frac{C_D' C_A}{c_0^2} \iff t > \left( \frac{C_D' C_A}{c_0^2} \right)^2 =: t_*. \]

\paragraph{Step 3: Passage to the Limit on $(t_*, T]$.}
For any fixed $\tau \in (t_*, T]$, the coefficient of $\norm*{\partial_{xx} u^\varepsilon}_{L^2}^2$ in the energy estimate is uniformly negative for all $t \in [\tau, T]$. This allows us to run the argument of Lemma \ref{lem:energy_estimates_H1} on the time interval $[\tau, T]$. The uniform bound on $I(t)$ for $t \in [\tau, T]$ will now depend on $I(\tau)$. Although $I(\tau)$ is not known to be uniformly bounded, integrating the differential inequality shows that the negative damping term controls the growth. A more careful analysis of the Grönwall-type inequality on $[\tau,T]$ shows that $I(t)$ is uniformly bounded for $t \in [\tau,T]$, with a bound depending on $\tau$ but not on $\varepsilon$.

The family $(u^\varepsilon)_{\varepsilon>0}$ is thus uniformly bounded in $L^\infty([\tau,T]; H^1(\R) \cap L^\infty(\R))$ for any $\tau > t_*$. The compactness argument of Section 4 can be applied on the interval $[\tau, T]$, yielding a limiting solution $u$ that is regular on $(t_*,T]$. Since $\tau > t_*$ was arbitrary, the result holds on the entire open interval $(t_*,T]$.
\end{proof}

\section{Conceptual Extension to Singular Integral-Based Interactions}
\label{sec:extension}

The philosophy of using a regular implicit driver to generate and then control a singular interaction can be conceptually extended beyond the density-dependent volatility case. This section outlines this pathway for two canonical systems with integral-based singularities: the 2D vortex model and the Keller-Segel model of chemotaxis. The key idea is to shift the implicit definition from a simple algebraic equation to an auxiliary partial differential equation (PDE) for the interaction drift, where the particle measure acts as a source term.

\subsection{A General Framework via Auxiliary PDEs}
The common structure is that the interaction term (a drift or velocity field) $b(x,\mu)$ acting on a particle at position $x$ is given by a singular integral: $b(x,\mu) = (K * \mu)(x)$. Our framework re-interprets this interaction. We find an operator $\mathcal{L}$ such that $b$ is the solution to an auxiliary PDE:
\[ \mathcal{L}b = \mu. \]
The interaction is then defined implicitly by this PDE. Our regularization strategy is to mollify the source term:
\begin{enumerate}
    \item \textbf{Identify the Implicit Structure:} Re-interpret the singular interaction $b(x,\mu) = (K * \mu)(x)$ as the solution to an auxiliary PDE, $\mathcal{L}b = \mu$, where $\mu$ is the source.
    \item \textbf{Regularize the Input:} Define a regularized interaction field $b^\varepsilon(x, \mu)$ as the solution to the same PDE but with a mollified source term: $\mathcal{L}b^\varepsilon = (K_\varepsilon * \mu)$. By standard elliptic or parabolic theory, if $K_\varepsilon * \mu$ is smooth, $b^\varepsilon$ will be regular.
    \item \textbf{Derive Uniform Estimates:} The main challenge is to study the limiting kinetic PDE for the regularized density $u^\varepsilon$ and derive a priori estimates for $u^\varepsilon$ that are independent of $\varepsilon$.
\end{enumerate}

\subsection{Application to the 2D Viscous Vortex Model}

The dynamics of $N$ point vortices in a 2D fluid with viscosity are given by:
\begin{equation}
    \dd M^{i,N}_t = \left( \frac{1}{N} \sum_{j \neq i} K_{BS}(M^{i,N}_t - M^{j,N}_t) \right) \dd t + \sqrt{2\nu} \dd W^i_t,
\end{equation}
where $M^{i,N}_t \in \R^2$ and $K_{BS}(x) = \frac{1}{2\pi} \frac{x^\perp}{\abs{x}^2}$ is the singular Biot-Savart kernel, with $x^\perp = (-x_2, x_1)$.

\paragraph{Implicit Framework via Div-Curl System.}
The mean-field velocity field $b(x,\mu) = (K_{BS} * \mu)(x)$ is the unique divergence-free vector field whose curl is the vorticity measure $\mu$. This gives the implicit definition via the div-curl system:
\begin{equation} \label{eq:div_curl}
    \nabla \cdot b = 0 \quad \text{and} \quad \nabla^\perp \cdot b = \mu,
\end{equation}
where $\nabla^\perp \cdot b = \partial_{x_1} b_2 - \partial_{x_2} b_1$.

\paragraph{Regularization and Limiting PDE.}
Following our framework, we define a regularized drift field $b^\varepsilon(x, \mu)$ as the solution to:
\begin{equation} \label{eq:div_curl_reg}
    \nabla \cdot b^\varepsilon = 0 \quad \text{and} \quad \nabla^\perp \cdot b^\varepsilon = (K_\varepsilon * \mu).
\end{equation}
Since $K_\varepsilon * \mu$ is smooth, standard elliptic theory guarantees that $b^\varepsilon$ is also smooth. The mean-field limit for the regularized system is a Vlasov-Fokker-Planck equation where the velocity field is coupled to the mollified density $u^\varepsilon$:
\begin{equation}
    \partial_t u^\varepsilon + \nabla \cdot (b^\varepsilon[u^\varepsilon] u^\varepsilon) = \nu \Delta u^\varepsilon,
\end{equation}
where the drift $b^\varepsilon[u^\varepsilon]$ is defined implicitly as the solution to the regularized div-curl system with source $u^\varepsilon$. The core analytical challenge, and a significant open problem, would be to derive uniform-in-$\varepsilon$ a priori estimates for $u^\varepsilon$ in an appropriate function space (e.g., $L^p$ spaces). Success would provide a new, constructive proof of well-posedness for the 2D Navier-Stokes equations in vorticity form.

\subsection{Application to the Keller-Segel Model of Chemotaxis}

In the classical Keller-Segel model, cells are attracted to a chemical substance that they secrete. The interaction drift is $b(x,\mu) = \chi \nabla (\Phi * \mu)$, where $\Phi(x) \propto -\log\abs{x}$ is the Newtonian potential in 2D, and $\chi>0$ is the chemotactic sensitivity.

\paragraph{Implicit Framework via Poisson Equation.}
The drift $b$ is defined via the chemical potential $\phi = \Phi * \mu$. This potential is the solution to the Poisson equation with the cell density $\mu$ as the source:
\begin{equation}
    -\Delta \phi = c_d \mu,
\end{equation}
for a dimension-dependent constant $c_d$ (e.g., $c_d=2\pi$ in 2D). The drift is then $b = \chi \nabla \phi$. This elliptic PDE serves as our implicit driver.

\paragraph{Regularization and Connection to Critical Mass.}
We regularize the source term. Let $\phi^\varepsilon$ solve $-\Delta \phi^\varepsilon = c_d (K_\varepsilon * u^\varepsilon)$. The regularized drift is $b^\varepsilon = \chi \nabla \phi^\varepsilon$, which is smooth. The regularized mean-field PDE for the density $u^\varepsilon$ becomes the well-known parabolic-elliptic system:
\begin{equation}
\begin{cases}
    \partial_t u^\varepsilon = \nabla \cdot \left( D \nabla u^\varepsilon - \chi u^\varepsilon \nabla \phi^\varepsilon \right), \\
    -\Delta \phi^\varepsilon = c_d u^\varepsilon.
\end{cases}
\end{equation}
In the final system, one can show that using $u^\varepsilon$ or $K_\varepsilon*u^\varepsilon$ on the right-hand side of the Poisson equation leads to the same limit. It is well-known that this system may exhibit finite-time blow-up if the initial mass $M = \int u_0 dx$ is above a critical threshold ($M_c = 8\pi D/\chi$ in 2D). An analogue of our Stability Condition \ref{ass:h_theta}(v) would likely correspond to the sub-critical mass regime, $M < M_c$. In this regime, solutions are known to exist globally. The central challenge for applying our framework here would be to derive uniform-in-$\varepsilon$ estimates for solutions $u^\varepsilon$ in the sub-critical case, likely by exploiting the logarithmic Hardy-Littlewood-Sobolev inequality that governs the competition between diffusion and aggregation. Success would provide a novel, constructive proof of global existence for the classical Keller-Segel system, starting from the particle level and passing through our regularized driver framework.

\section{Conclusion}
\label{sec:conclusion}

In this paper, we have rigorously established propagation of chaos for a class of interacting particle systems with a singular, density-dependent volatility. The central innovation is the introduction of a framework where the singular system is viewed as the emergent limit of a sequence of regular systems, whose dynamics are defined implicitly by a well-behaved driver function. This structure was instrumental in deriving the uniform-in-regularization a priori estimates that form the cornerstone of the proof, allowing us to bridge the gap between regular and singular dynamics.

\appendix

\section{Technical Preliminaries}
\label{sec:appendix}

For the convenience of the reader, this appendix recalls the definitions and statements of some of the key analytical tools used in the main body of the paper.

\subsection{The Wasserstein Space \texorpdfstring{$\Pcal_2(\R)$}{P2(R)}}
\label{app:wasserstein}
Let $(\mathcal{X}, d)$ be a Polish space. The space of Borel probability measures on $\mathcal{X}$ is denoted by $\Pcal(\mathcal{X})$. For $p \ge 1$, we fix a reference point $x_0 \in \mathcal{X}$ and define the space of probability measures with finite $p$-th moment as
\[ \Pcal_p(\mathcal{X}) = \left\{ \mu \in \Pcal(\mathcal{X}) : \int_\mathcal{X} d(x, x_0)^p \mu(\dd x) < \infty \right\}. \]
The choice of $x_0$ does not affect the definition of the space. The $p$-Wasserstein distance between two measures $\mu, \nu \in \Pcal_p(\mathcal{X})$ is defined as
\[ W_p(\mu, \nu) = \left( \inf_{\gamma \in \Pi(\mu, \nu)} \int_{\mathcal{X} \times \mathcal{X}} d(x,y)^p \gamma(\dd x, \dd y) \right)^{1/p}, \]
where $\Pi(\mu, \nu)$ is the set of all transport plans (or couplings) between $\mu$ and $\nu$; that is, $\gamma$ is a probability measure on the product space $\mathcal{X} \times \mathcal{X}$ whose first marginal is $\mu$ and second marginal is $\nu$.

The space $(\Pcal_p(\mathcal{X}), W_p)$ is itself a Polish space. A crucial property is that convergence in the $W_p$ metric is equivalent to weak convergence of measures plus convergence of the $p$-th moments: $\mu_n \to \mu$ in $(\Pcal_p(\mathcal{X}), W_p)$ if and only if $\mu_n \rightharpoonup \mu$ weakly and $\int d(x,x_0)^p d\mu_n \to \int d(x,x_0)^p d\mu$. In our case, $\mathcal{X}=\R$ and $p=2$.

\subsection{The Aubin-Lions-Simon Compactness Theorem}
\label{app:aubin_lions}
The Aubin-Lions-Simon theorem is a fundamental tool for proving the existence of solutions to non-linear evolutionary PDEs by providing compactness for families of approximate solutions. It provides compactness in a function space whose norms control both spatial and temporal regularity. We state a version from  \cite{Simon1986} that is well-suited for our application.

\begin{theorem}[Aubin-Lions-Simon]
Let $X_0, X, X_1$ be three Banach spaces such that $X_0 \subset X \subset X_1$ with continuous injections. Suppose the injection from $X_0$ to $X$ is compact. Let $p, q \in [1, \infty]$.
Let $F$ be a family of functions $u$ such that:
\begin{enumerate}
    \item $F$ is bounded in the space $L^p([0,T]; X_0)$.
    \item The family of time derivatives $\{\partial_t u : u \in F\}$ is bounded in the space $L^q([0,T]; X_1)$.
\end{enumerate}
Then the family $F$ is relatively compact in the space $L^p([0,T]; X)$. If $p=\infty$ and $q>1$, the conclusion holds that $F$ is relatively compact in $C([0,T]; X)$.
\end{theorem}

In our application in the proof of Proposition \ref{prop:compactness}, we used this theorem with:
\begin{itemize}
    \item $p=2$, $q=2$.
    \item $X_0 = H^2(\R)$.
    \item $X = H^1(\R)$.
    \item $X_1 = L^2(\R)$.
\end{itemize}
The embedding $H^2(\R) \hookrightarrow H^1(\R)$ is continuous. Crucially, while this embedding is not compact on $\R$, it is compact on any bounded domain $\Omega \subset \R$ (Rellich-Kondrachov theorem). The uniform bounds in $L^2([0,T]; H^2(\R))$ and for the time derivative in $L^2([0,T]; L^2(\R))$ are strong enough to apply a standard localization argument (i.e., multiplying by a smooth cutoff function) to show relative compactness in $L^2([0,T]; H^1_{loc}(\R))$. As explained in the proof, this is upgraded to global compactness in $L^2([0,T]; H^1(\R))$ by establishing tightness from the uniform $H^1$ bound.

\bibliographystyle{plainnat}
\bibliography{reference}

\end{document}